\documentclass[12pt]{amsart}
\usepackage{enumerate}
\usepackage[margin=1in]{geometry}

\newtheorem{theorem}{Theorem}[section]
\newtheorem{lemma}[theorem]{Lemma}
\newtheorem{proposition}[theorem]{Proposition}

\theoremstyle{definition}

\theoremstyle{remark}

\newtheorem{conjecture}[theorem]{Conjecture}
\numberwithin{equation}{section}
\newcommand{\EE}{\mathbb{E}}
\DeclareMathOperator{\tr}{tr}
\DeclareMathOperator{\BBeta}{Beta}

\title{Hardy-type norms of matrices}

\author{Leonid V. Kovalev}
\address{215 Carnegie, Department of Mathematics, Syracuse University, Syracuse, NY 13244, USA \\
ORCID: 0000-0001-8002-7155}
\email{lvkovale@syr.edu}



\begin{document}

\begin{abstract}
We show that the function that assigns to each square matrix $A$ the geometric mean of $|Ax|$ over all unit vectors $x$ is a norm. The same holds for $L^p$ means with $0<p<1$. Some properties of these norms are proved, and others are conjectured. 
\end{abstract}

\maketitle

\section{Introduction}

Howe and Johnson~\cite{HoweJohnson} considered several ``expected-value norms'' of matrices, which are defined as the average stretch factor $|Ax|/|x|$ rather than the maximal one. 
Such norms are generally not submultiplicative, so are not true matrix norms. But they are clearly subadditive, as the arithmetic mean of convex functions is convex. 

When averaging multiplicative increments, such as percentage changes or $|Ax|/|x|$ above, it is  natural to use the geometric mean: this was forcefully argued in~\cite{FlemingWallace}, see also~\cite{Carvalho}.
Accordingly, we define the \textit{geometric-mean norm} of a real $n\times n$ matrix $A$ by 
\begin{equation}\label{H0-def-intro}
\|A\|_{H^0} = \exp\left(\int_{S^{n-1}} \log |Ax| \,d\mu(x)\right)    
\end{equation}
where $\mu$ is the normalized surface measure of the unit sphere $S^{n-1}$ and $|\cdot |$ is the Euclidean norm. Since the geometric mean does not preserve convexity in general, it is not obvious whether $\|A\|_{H^0}$ is indeed a norm. Theorem~\ref{thm-h0c} shows that it is. The fact that the Euclidean vector norm is used in~\eqref{H0-def-intro} is essential here. 

Between~\eqref{H0-def-intro} and the arithmetic-mean norms of Howe and Johnson lie the Hardy-type norms 
\begin{equation}\label{Hp-def-intro}
\|A\|_{H^p} = \left(\int_{S^{n-1}} |Ax|^p \,d\mu(x)\right)^{1/p},\quad 0<p<\infty.    
\end{equation}
Indeed, one obtains~\eqref{H0-def-intro} when $p\to 0$ and the arithmetic mean when $p=1$. The subadditivity of~\eqref{Hp-def-intro} is clear for $p\ge 1$. Theorem~\ref{thm-hp-norm} shows that subadditivity holds even when $0<p<1$, answering a question from~\cite[p.~13]{KovalevYang}.

One reason to study the Hardy norms of matrices is their connection with harmonic mappings in $\mathbb R^n$. Specifically, a version of the Schwarz lemma for harmonic mappings $F\colon B^n\to B^n$ on the unit ball $B^n$ describes the possible values of the derivative matrix $DF(0)$ in terms of the dual of the $H^1$ norm~\cite[Theorem 5.1]{KovalevYang}. This dual norm, which is difficult to evaluate directly, is remarkably close to the $H^4$ norm: in two dimensions, this was shown in~\cite{Brevig-etal, KovalevYang}. A reader who is not interested in the geometric mean or $p$-means with $p<1$ is encouraged to skip forward to \S\ref{sec:duality} and ponder the inequality 
$\frac{1}{n}\tr\left(A^{\mathsf T}B\right) 
\le \|A\|_{H^4} \|B\|_{H^1}$ ($A, B\in \mathbb R^{n\times n}$) which is conjectured to hold for $n>2$ even though it fails for $n=2$. 

\subsection*{Generative AI disclosure.} The author used ChatGPT by OpenAI and Claude by Anthropic during all stages of this work. The final version of the text was written by the author, who is responsible for its correctness. 

\section{Preliminaries}

Let $\mathbb R^n_+ = \{x\in \mathbb R^n \colon \forall k\ x_k > 0\}$ be the positive orthant. A function $\Phi$ is $1$-homogeneous on $\mathbb R^n_+$ if $\Phi(tx) = t \Phi(x)$ for all $t>0$ and $x\in \mathbb R^n_+$. The following ``folklore'' lemma will be repeatedly used. 

\begin{lemma}\label{lem-mixed-partials}
Suppose that $\Phi\colon \mathbb R^n_+\to\mathbb R$ is a $C^2$-smooth $1$-homogeneous function such that 
\begin{equation}\label{eq-lem-partials}
\frac{\partial^2 \Phi }{\partial t_i \partial t_j} \le 0,\quad 1\le i<j\le n, \quad t\in \mathbb R^n_+.
\end{equation}
Then $\Phi$ is convex. 
\end{lemma}

\begin{proof} Let $H$ be the Hessian matrix of $\Phi$. 
Differentiating Euler's identity 
$\Phi(t) = \sum_{i=1}^n t_i\, \partial \Phi/\partial t_i$ with respect to $t_j$ and simplifying, we obtain 
$\sum_{i=1}^n t_i\, H_{ij} = 0$. The algebraic identity 
\[
\xi^T H \xi = -\sum_{i<j} H_{ij}\frac{(\xi_i t_j-\xi_j t_i)^2}{t_it_j},\quad \xi\in \mathbb R^n
\]
demonstrates that $H$ is positive semidefinite. Thus $\Phi$ is convex.      
\end{proof}

The converse of Lemma~\ref{lem-mixed-partials} is false in general: a convex $1$-homogeneous function may have positive mixed second derivatives. However, it holds in two dimensions. 

\begin{lemma}\label{lem-mixed-partials-2d}
Suppose that $\Phi\colon \mathbb R^2_+\to\mathbb R$ is a $C^2$-smooth convex $1$-homogeneous function. Then
\begin{equation}\label{eq-lem-partials-2d}
\frac{\partial^2 \Phi }{\partial t_1 \partial t_2} \le 0\quad \text{on } \mathbb R^2_+.
\end{equation}
\end{lemma}

\begin{proof} As in the proof of Lemma~\ref{lem-mixed-partials}, we have 
$t_1H_{11} + t_2 H_{21} = 0$. Since $\Phi$ is convex, $H_{11}\ge 0$, which implies $H_{21} \le 0$.    
\end{proof}

\section{The geometric mean norm}

Given a real $n\times n$ matrix $A$, we write $\EE_x \log |Ax|$ for the integral of $\log |Ax|$ with respect to the normalized spherical measure. In other words, $x$ is a uniformly distributed random unit vector. 
It is easy to see that $\EE_x \log |Ax|$ is finite unless $A$ is the zero matrix, in which case it is $-\infty$. Recall that 
\begin{equation}\label{H0-def}
\|A\|_{H^0}  = \exp\left(\EE_x \log |Ax|\right).    
\end{equation}

\begin{theorem}\label{thm-h0c} For every $n\ge 2$, the geometric mean~\eqref{H0-def} is an orthogonally invariant norm on $\mathbb R^{n\times n}$. 
\end{theorem}

Although the above theorem can be obtained by letting $p\to 0$ in Theorem~\ref{thm-hp-norm}, the direct proof is shorter and more transparent. 

\begin{proof} It is easy to see that~\eqref{H0-def} is invariant under multiplication of $A$ by orthogonal matrices on either side. Therefore, it depends only on the singular values of $A$. 

First consider the case of a diagonal matrix $A$ with positive diagonal entries $a_1,\dots, a_n$. The function 
$\Phi(a):=\|A\|_{H^0}$ is evidently smooth and $1$-homogeneous on $\mathbb R^n_+$. In order to apply Lemma~\ref{lem-mixed-partials}, we need to show that the mixed second partials of $\Phi$ are nonpositive. 

Let $\sigma(A, x) = \sum_{k=1}^n a_k^2 x_k^2$ and $\varphi(a) = \log \Phi(a) = \frac{1}{2}\EE_x\log \sigma(A, x)$. 
The relevant partial derivatives of $\varphi$ are
\[
\frac{\partial \varphi}{\partial a_i}
= a_i\, \EE_x \frac{x_i^2}{\sigma(A, x)} \quad  \text{and} \quad
\frac{\partial^2 \varphi}{\partial a_i\partial a_j}
= - 2a_ia_j\, \EE_x \frac{x_i^2 x_j^2}{\sigma(A, x)^2}  \quad (i<j).
\]
The Hessian matrix of $\Phi$ is
$D^2 (e^\varphi) = e^{\varphi} \left(D^2\varphi + \nabla \varphi (\nabla \varphi)^T\right)$. Thus, in order to apply Lemma~\ref{lem-mixed-partials} we need to show 
\begin{equation}\label{eq-need-for-h0}
\left(\EE_x \frac{x_i^2}{\sigma(A, x)}\right) \, \left(\EE_x \frac{x_j^2}{\sigma(A, x)}\right) \le 2\, \EE_x \frac{x_i^2 x_j^2}{\sigma(A, x)^2},\quad i<j.  
\end{equation}

Assume~\eqref{eq-need-for-h0} for now. By Lemma~\ref{lem-mixed-partials}, $\Phi$ is a convex function on the positive orthant of $\mathbb R^n$. The formula for $\Phi$ extends continuously from the positive orthant to its closure $[0,\infty)^n$. Moreover, $\Phi$ is $1$-homogeneous and increasing in each variable. Define its extension to $\mathbb R^n$ by $\Phi(a) = \Phi(\tilde{a})$ where $\tilde a_k = |a_k|$, $k=1,\dots, n$. The monotonicity and convexity of $\Phi$ on $[0,\infty)^n$ imply that for all $a,b\in \mathbb R^n$ and all $0\le t\le 1$ we have 
\[
\Phi(ta+(1-t)b) \le \Phi\left(t \tilde a + (1-t)\tilde b\right)
\le t\Phi(a) + (1-t)\Phi(b).
\]
Thus $\Phi$ is convex on $\mathbb R^n$. It is also $1$-homogeneous and vanishes only at the origin, so it is a norm on $\mathbb R^n$. This norm is invariant under permutations and sign changes of the coordinates; in other words, it is a symmetric gauge function. By von Neumann's theorem (see, for example, \cite[Theorem IV.2.1]{Bhatia} or 
\cite[Theorem 7.4.7.2]{HornJohnson}), applying $\Phi$ to the vector of singular values defines an orthogonally invariant norm on $\mathbb R^{n\times n}$. By the singular value decomposition, this norm is precisely the quantity defined in~\eqref{H0-def}. The proof of Theorem~\ref{thm-h0c} is complete, modulo the inequality~\eqref{eq-need-for-h0}.
\end{proof}

In order to prove~\eqref{eq-need-for-h0},  we need a few integral identities. 

\begin{lemma}\label{lemma-trig-integral} For $\alpha, \beta>0$ let 
$D(\theta)= \alpha\cos^2\theta+\beta\sin^2\theta$. Then 
\begin{equation}\label{eq-trigint-0}
\frac{1}{2\pi}\int_0^{2\pi} \log D(\theta)\,d\theta = 2\log \frac{\sqrt{\alpha}+\sqrt{\beta}}{2}
\end{equation}
\begin{equation}\label{eq-trigint-1}
\frac{1}{2\pi}\int_0^{2\pi} \frac{\cos^2 \theta}{D(\theta)}\,d\theta = \frac{1}{\sqrt{\alpha}\,(\sqrt{\alpha}+\sqrt{\beta})} 
= \frac{1-\sqrt{\beta/\alpha}}{\alpha-\beta}
\end{equation}
\begin{equation}\label{eq-trigint-2}
\frac{1}{2\pi}\int_0^{2\pi} \frac{\cos^2 \theta \sin^2\theta }{D(\theta)^2}\,d\theta = \frac{1}{2\sqrt{\alpha \beta}\,(\sqrt{\alpha}+\sqrt{\beta})^2}
= \frac{\sqrt{\beta/\alpha} + \sqrt{\alpha/\beta} - 2}{2(\alpha-\beta)^2}
\end{equation}
where the last parts of~\eqref{eq-trigint-1}--\eqref{eq-trigint-2} hold when $\alpha\ne \beta$. 
\end{lemma}

\begin{proof} Observe that 
$
D(\theta) = \frac{1}{4} \left|(\sqrt{\alpha}+\sqrt{\beta})e^{i\theta} + 
(\sqrt{\alpha}-\sqrt{\beta}) e^{-i\theta}\right|^2
$, hence 
\begin{equation}\label{eq-log-aux}
\log D(\theta) = 2\log \frac{\sqrt{\alpha}+\sqrt{\beta}}{2} 
+ 2\log \left|1 + \frac{\sqrt{\alpha}-\sqrt{\beta}}{\sqrt{\alpha}+\sqrt{\beta}} e^{-2i\theta}\right|.
\end{equation}
The second term on the right in~\eqref{eq-log-aux} has zero mean because the function $z\mapsto \log |1+cz|$ is harmonic on the unit disk when $|c|<1$. This proves~\eqref{eq-trigint-0}. The identity~\eqref{eq-trigint-1} is obtained by differentiating~\eqref{eq-trigint-0} with respect to $\alpha$, and~\eqref{eq-trigint-2} follows by additional differentiation in $\beta$.     
\end{proof}

\begin{proof}[Proof of~\eqref{eq-need-for-h0}] By relabeling the indices, we can take $i=n-1$ and $j=n$, reducing~\eqref{eq-need-for-h0} to 
\begin{equation}\label{eq-moments1b}
\left(\EE_x \frac{x_{n-1}^2}{\sigma(A, x)} \right) 
\left(\EE_x \frac{x_n^2}{\sigma(A, x)} \right)  \le 
2 \, \EE_x \frac{x_{n-1}^2 x_n^2}{\sigma(A, x)^2} .
\end{equation}    
Since both sides of~\eqref{eq-moments1b} are  continuous with respect to $\{a_k\}$, it suffices to consider the case $a_{n-1}\ne a_n$. The orthogonal projection $y=P(x)=(x_1,\dots, x_{n-2})\in \mathbb R^{n-2}$ pushes the normalized surface measure on $S^{n-1}$ to the normalized Lebesgue measure on the $(n-2)$-dimensional ball $B^{n-2}$. Thus, the average of an integrable function $f\colon S^{n-1}\to \mathbb R$ can be computed as 
\begin{equation}\label{eq-moments2}
\EE_x f = \EE_y \left(\frac{1}{2\pi} \int_0^{2\pi} f(y_1,\dots, y_{n-2}, r\cos \theta, r\sin\theta)\,d\theta \right) 
\end{equation}
where $y$ is uniformly distributed on $B^{n-2}$ and 
$r=\sqrt{1-|y|^2}$. For each of the three terms in~\eqref{eq-moments1b} the corresponding inner integral in~\eqref{eq-moments2} can be computed with the help of Lemma~\ref{lemma-trig-integral}, as follows. Let  
$g(y) = r^{-2} \sum_{k=1}^{n-2} a_k^2 y_k^2$ 
and rewrite $\sigma(A, x)$ as 
\begin{equation}\label{eq-rewrite-sigma}
\sigma(A, x) = r^2 \left( (a_{n-1}^2 + g(y)) \cos^2 \theta + 
(a_{n}^2 +g(y)) \sin^2 \theta \right). 
\end{equation}
Plugging~\eqref{eq-rewrite-sigma} into each of the three denominators in~\eqref{eq-moments1b}, we obtain expressions to which  Lemma~\ref{lemma-trig-integral} applies. Namely,~\eqref{eq-trigint-1} and~\eqref{eq-moments2} yield
\begin{equation}\label{eq-moments3}
\EE_x \frac{x_{n-1}^2}{\sigma(A, x)}  
= \EE_x \frac{\cos^2\theta }{(a_{n-1}^2 + g(y)) \cos^2 \theta + 
(a_{n}^2 +g(y)) \sin^2 \theta}    
= \frac{1 - \EE_y h(y)}{a_{n-1}^2-a_{n}^2}
\end{equation}
where 
\[ h(y)=\sqrt{\frac{a_n^2 + g(y)}{a_{n-1}^2 + g(y)}}.\]
Exchanging the roles of $x_{n-1}$ and $x_n$, we find that 
\begin{equation}\label{eq-moments4}
\EE_x \frac{x_{n}^2}{\sigma(A, x)}    
= \frac{\EE_y (1/h(y)) - 1}{a_{n-1}^2-a_{n}^2}.
\end{equation}
In a similar way,~\eqref{eq-trigint-2} yields 
\begin{equation}\label{eq-moments5}
2 \, \EE_x \frac{x_{n-1}^2 x_n^2}{\sigma(A, x)^2}  = 
\frac{\EE_y h(y) 
+  \EE_y (1/h(y))  - 2}{(a_{n-1}^2-a_{n}^2)^2}. 
\end{equation}
Combining~\eqref{eq-moments3}--\eqref{eq-moments5} we obtain 
\begin{equation}\label{eq-moments6}
2 \, \EE_x \frac{x_{n-1}^2 x_n^2}{\sigma(A, x)^2} - \left(\EE_x \frac{x_{n-1}^2}{\sigma(A, x)} \right) 
\left(\EE_x \frac{x_n^2}{\sigma(A, x)} \right)
= \frac{\EE_y h(y) \cdot \EE_y (1/h(y)) -1}{(a_{n-1}^2-a_{n}^2)^2}   \ge 0
\end{equation}
where the last step is the Cauchy-Schwarz inequality. This completes the proof of~\eqref{eq-moments1b}.    
\end{proof}


\section{Comparison of the geometric mean and other matrix norms} 

In dimension $n=2$ the geometric-mean norm coincides with the normalized trace norm, i.e., the arithmetic mean of singular values. This can be inferred from~\eqref{eq-trigint-0} and is a special case of~\eqref{eq-trace-comparison} below. 

For $n>2$, the geometric-mean norm can rarely be obtained in a closed form. However, we have such a form for orthogonal projection matrices. 

\begin{proposition}\label{prop-projection-H0}
Let $1\le k\le n-1$, and let $P_k$ be the matrix of an orthogonal projection of rank $k$ in $\mathbb R^n$. Then
\begin{equation}\label{eq-proj-H0}
\|P_k\|_{H^0}
= \exp\left(\frac{\psi(k/2)-\psi(n/2)}{2}\right)
\end{equation}
where $\psi=\Gamma'/\Gamma$ is the digamma function.
\end{proposition}

\begin{proof}  We may assume that $P_k$ is the projection onto the first $k$ coordinates. Let $g=(g_1,\dots,g_n)$ be the standard Gaussian on $\mathbb R^n$, so that $x=g/|g|$ is uniformly distributed on $S^{n-1}$. Then 
\[
|P_k x|^2=\frac{g_1^2+\dots+g_k^2}{g_1^2+\dots+g_n^2}=\frac{G_1}{G_1+G_2}
\]
where $G_1$ and $G_2$ are $\chi^2$-variables with $k$ and $(n-k)$ degrees of freedom, respectively. Since $G_1$ and $G_2$ are independent, the ratio $G_1/(G_1+G_2)$ follows the distribution $\BBeta(k/2, (n-k)/2)$, see~\cite[\S25.2]{JKB2}. The logarithmic moment of this Beta distribution is $\psi(k/2)-\psi(n/2)$, see~\cite[\S25.3]{JKB2}. Thus, $\EE_x\log (|P_k x|^2) = \psi(k/2)-\psi(n/2)$ which yields~\eqref{eq-proj-H0}.
\end{proof}

In practice, it is easier to evaluate~\eqref{eq-proj-H0} in terms of harmonic numbers 
\[
H_\alpha := \int_0^1 \frac{1-x^\alpha}{1-x}\,dx,\quad 
\alpha> -1,
\]
using the relation $\psi(\alpha+1) = H_\alpha - \gamma$, where $\gamma$ (Euler's constant) cancels out of~\eqref{eq-proj-H0}. That is, 
\[
\|P_k\|_{H^0} = \exp\left(\frac{H_{(k-2)/2}-H_{(n-2)/2}}{2}\right). 
\]
The property $H_\alpha = H_{\alpha-1} + \frac{1}{\alpha}$ together with $H_0=0$ and $H_{-1/2} = -2\log 2$ determine all the harmonic numbers needed for $\|P_k\|_{H^0}$. For example, in $n=3$ dimensions we have $\|P_1\|_{H^0}=1/e$ and $\|P_2\|_{H^0} = 2/e$. Also, 
\[
\frac{\|P_2\|_{H^0}}{\|P_1\|_{H^0}}
= \exp\left(\frac{H_0-H_{-1/2}}{2}\right) = 2
\]
in any dimension $n\ge 2$. 

Let us write $\kappa_n$ for the value $\|P_1\|_{H^0} = \frac{1}{2}\exp\left(-\frac12 H_{(n-2)/2}\right)$. This constant appears in a sharp inequality  between the $H^0$ norm and the operator norm $\|\cdot\|_\text{op}$. Namely, for any $A\in \mathbb R^{n\times n}$ we have
\begin{equation}\label{eq-op-comparison}
\kappa_n \|A\|_\text{op}  
\le \|A\|_{H^0} \le \|A\|_\text{op}. 
\end{equation}
Indeed, writing the singular values of $A$ as $\sigma_1\ge \dots \ge \sigma_n$ and recalling  
$\|A\|_{H^0} = \Phi(\sigma_1,\dots, \sigma_n)$ from the proof of Theorem~\ref{thm-h0c}, we obtain~\eqref{eq-op-comparison} from 
\[
\sigma_1 \|P_1\|_{H^0} =
\Phi(\sigma_1,0, \dots, 0) 
\le \Phi(\sigma_1,\sigma_2,\dots, \sigma_n) 
\le \Phi(\sigma_1,\sigma_1,\dots, \sigma_1)  
=\sigma_1 \|I\|_{H^0} = \sigma_1.
\]

A tighter comparison is available in terms of the trace norm 
$\|A\|_*=\sigma_1+\dots+\sigma_n$, namely 
\begin{equation}\label{eq-trace-comparison}
\frac{1}{n} \|A\|_*  
\le \|A\|_{H^0} \le \kappa_n   \|A\|_* 
\end{equation}
(note that $\kappa_2 = 1/2$ makes~\eqref{eq-trace-comparison} an identity when $n=2$). 
The inequality~\eqref{eq-trace-comparison} is equivalent to 
\begin{equation}\label{eq-trace-comparison-2}
\bar\sigma 
\le \Phi(\sigma_1,\sigma_2,\dots, \sigma_n) \le n\kappa_n \bar\sigma, \quad \text{where }
\bar\sigma = \frac{\sigma_1+\dots+\sigma_n}{n}.
\end{equation} 
To prove the first half of~\eqref{eq-trace-comparison-2}, note that averaging $\Phi$ over all permutations of $(\sigma_1,\dots,\sigma_n)$ leaves $\Phi(\sigma_1,\sigma_2,\dots, \sigma_n)$ unchanged, while the value of $\Phi$ at the average of these permuted vectors is $\Phi(\bar\sigma,\dots,\bar\sigma)=\bar\sigma$. The second half of~\eqref{eq-trace-comparison-2} follows by writing $(\sigma_1,\dots,\sigma_n)$ as the linear combination of standard basis vectors and using the subadditivity of $\Phi$.  

The constant $\kappa_n$ is also related to the question of submultiplicativity. Since $P_1$ is idempotent and $\|P_1\|_{H^0}=\kappa_n<1$, the $H^0$ norm is not submultiplicative. Instead, it has the property 
\[
\|AB\|_{H^0} \le \|A\|_{\text{op}} \|B\|_{H^0} 
\]
by virtue of being orthogonally invariant~\cite[Proposition IV.2.4]{Bhatia}. However, the rescaled norm
$A \mapsto  \kappa_n^{-1} \|A\|_{H^0}$ 
is submultiplicative (i.e., is a matrix norm) by~\cite[Theorem 7.4.10.1]{HornJohnson}. 


\section{The Hardy-type norms}

For $0<p<\infty$, the Hardy ($H^p$) mean of a real $n\times n$ matrix $A$ is defined by
\begin{equation}\label{Hp-def}
\|A\|_{H^p}
   =\left(\int_{S^{n-1}}|Ax|^p\,d\mu(x)\right)^{1/p}
   =\left(\EE_x|Ax|^p\right)^{1/p}.
\end{equation}
For $p\ge 1$ the quantity~\eqref{Hp-def} is obviously a norm, being the restriction of the $L^p$ norm to the subspace of $L^p(S^{n-1}; \mathbb R^n)$ formed by linear maps. For $0<p<1$ the space $L^p(S^{n-1}; \mathbb R^n)$ is no longer normed, but its subspace of linear maps still is. 

\begin{theorem}\label{thm-hp-norm}
For every $n\ge 2$ and $0<p<\infty$, the quantity defined by~\eqref{Hp-def} is an orthogonally invariant norm on $\mathbb R^{n\times n}$. 
\end{theorem}

\begin{proof} As noted above, only the case $0<p<1$ needs to be treated. Consider the case of a diagonal matrix $A$ with positive diagonal entries $a_1,\dots, a_n$. The function 
$\Phi(a):=\|A\|_{H^p}$ is evidently smooth and $1$-homogeneous on $\mathbb R^n_+$. In order to apply Lemma~\ref{lem-mixed-partials}, we need to show that the mixed second partials of $\Phi$ are nonpositive. 

Let $q=p/2$, $\sigma = \sigma(A, x) = \sum_{k=1}^n a_k^2 x_k^2$ and $\varphi(a) =  \Phi(a)^{p} = \EE_x \sigma^q$. 
The relevant partial derivatives of $\varphi$ are
\[
\frac{\partial \varphi}{\partial a_i}
= p a_i\, \EE_x \left(x_i^2 \sigma^{q-1}\right) \quad  \text{and} \quad
\frac{\partial^2 \varphi}{\partial a_i\partial a_j}
= p(p-2) a_ia_j\, \EE_x \left(x_i^2 x_j^2 \sigma^{q-2}\right)  \quad (i<j).
\]
The Hessian matrix of $\Phi$ is
\[D^2 (\varphi^{1/p}) = \varphi^{1/p-2} \left(\frac{\varphi}{p}\,  D^2\varphi + \frac{1-p}{p^2}\, \nabla \varphi (\nabla \varphi)^T\right).\] 
In order to apply Lemma~\ref{lem-mixed-partials} we need to show $\partial^2 \Phi/\partial a_i \partial a_j\le 0$, which amounts to 
\begin{equation}\label{eq-need-for-hp}
\EE_x \left(x_i^2 \sigma^{q-1}\right) \, \EE_x \left(x_j^2 \sigma^{q-1}\right) \le \frac{2-p}{1-p}\, 
\EE_x \left(\sigma^q\right)
\EE_x \left(x_i^2 x_j^2 \sigma^{q-2}\right) ,\quad i<j.   
\end{equation} 
Once~\eqref{eq-need-for-hp} is proved, Theorem~\ref{thm-hp-norm} follows in the same way as Theorem~\ref{thm-h0c}. We thus proceed to the proof of~\eqref{eq-need-for-hp}. By relabeling the variables, we can take $i=n-1$ and $j=n$. 
In place of Lemma~\ref{lemma-trig-integral} the proof involves a nonelementary integral representing the $2\times 2$ case of the $H^p$ norm. Specifically, Theorem~2.1 in~\cite{KovalevYang} shows that the function 
\[
F(s,t) := \left(\frac1{2\pi}\int_0^{2\pi}
\bigl(s^2\cos^2\theta+t^2\sin^2\theta\bigr)^{p/2}
 \,d\theta \right)^{1/p}
\]
is convex with respect to $(s,t) \in \mathbb R^2$. Since $F$ is $1$-homogeneous, by Lemma~\ref{lem-mixed-partials-2d} we have $\partial^2 F/\partial s\partial t \le 0$ for $s,t>0$. Introduce the notation $\EE_\theta f= \frac{1}{2\pi}\int_0^{2\pi} f \,d\theta$ and $T = s^2\cos^2\theta+t^2\sin^2\theta$, so that $F = \left(\EE_\theta T^q \right)^{1/p}$. Writing out the inequality $\partial^2 F/\partial s\partial t \le 0$, we obtain 
\begin{equation}\label{eq-from-mixed-partial}
\EE_\theta (T^{q-1} \cos^2\theta )\,
\EE_\theta (T^{q-1} \sin^2\theta )
\le \frac{2-p}{1-p}\,
\EE_\theta (T^{q})
\, \EE_\theta (T^{q-2} \cos^2\theta \sin^2\theta )
\end{equation}
which strongly resembles~\eqref{eq-need-for-hp} and indeed will lead us there.
 
As in the proof of~\eqref{eq-need-for-h0}, we use polar coordinates to express the average of a function $f\in L^1(S^{n-1})$ as 
\begin{equation}\label{eq-wm-fiber}
\EE_x f=\EE_y\left(\EE_\theta f(y_1,\dots,y_{n-2},r\cos\theta,r\sin\theta)\right)
\end{equation}
where $y$ is uniformly distributed on $B^{n-2}$ and $r=\sqrt{1-|y|^2}$. Write $g=r^{-2}\sum_{k=1}^{n-2}a_k^2 y_k^2$ and plug 
\begin{equation}\label{eq-sigma-from-g}
T = (a_{n-1}^2+g)\cos^2\theta+(a_n^2+g)\sin^2\theta
\end{equation}
into~\eqref{eq-from-mixed-partial}, expressing the resulting inequality as 
\begin{equation}\label{eq-theta-only}
\frac{\EE_\theta\, (T^{q-1} \cos^2\theta)}{\EE_\theta T^{q}}\, 
\frac{\EE_\theta\, (T^{q-1} \sin^2\theta)}{\EE_\theta T^{q}}\
\le \frac{2-p}{1-p}\,
\frac{\EE_\theta \,(T^{q-2} \cos^2\theta \sin^2\theta )}{\EE_\theta T^{q}}. 
\end{equation}
Each quotient on the left in~\eqref{eq-theta-only} depends on $y$ only through the scalar $g=g(y)$ appearing in $T$. Since $dT/dg = 1$ and $0<q<1$, each numerator in ~\eqref{eq-theta-only} is decreasing with respect to $g$ and each denominator is increasing. Thus we can write 
\begin{equation*} \begin{split}
u(g) &= \frac{\EE_\theta\, (T^{q-1} \cos^2\theta)}{\EE_\theta T^{q}} = \frac{\EE_\theta\, (x_{n-1}^2 \sigma^{q-1} )}{\EE_\theta \sigma^{q}} \\ 
v(g) &= \frac{\EE_\theta\, (T^{q-1} \sin^2\theta)}{\EE_\theta T^{q}} 
= \frac{\EE_\theta\, (x_{n}^2 \sigma^{q-1} )}{\EE_\theta \sigma^{q}}
\end{split}\end{equation*}
where $u$ and $v$ are some nonincreasing real functions. The weighted form of Chebyshev's association inequality~\cite[Theorem 2.14]{BLM} states that 
\begin{equation}\label{eq-association}
\EE[Wu(X)]\,\EE[Wv(X)]   \le \EE[W]\,\EE[Wu(X)v(X)] 
\end{equation}
for any random variables $X$ and $W$, as long as $W\ge 0$ and the moments exist. 
Using~\eqref{eq-association} with $X=g(y)$ and $W=\EE_\theta \sigma^q$, and recalling that $\EE_y \EE_\theta(\cdot)=\EE_x(\cdot)$, we obtain
\begin{equation}\label{eq-from-cheb}
\EE_x (x_{n-1}^2 \sigma^{q-1}) 
\EE_x (x_{n}^2 \sigma^{q-1}) \le 
\left(\EE_x \sigma^q \right)
\EE_y \left(\frac{\EE_\theta\, (x_{n-1}^2 \sigma^{q-1}) \EE_\theta\, (x_{n}^2 \sigma^{q-1})}{\EE_\theta \sigma^q}\right). 
\end{equation} 
Recasting~\eqref{eq-theta-only} in the form
\[
\frac{\EE_\theta\, (x_{n-1}^2 \sigma^{q-1}) \EE_\theta\, (x_{n}^2 \sigma^{q-1})}{\EE_\theta \sigma^q} \le 
\frac{2-p}{1-p}\,\EE_\theta \,(x_{n-1}^2 x_n^2 \sigma^{q-2}) 
\]
and plugging it into~\eqref{eq-from-cheb} yields 
\begin{equation*}
\EE_x (x_{n-1}^2 \sigma^{q-1}) 
\EE_x (x_{n}^2 \sigma^{q-1}) 
\le \frac{2-p}{1-p}\, \EE_x (\sigma^q) \, 
\EE_x \,(x_{n-1}^2 x_n^2 \sigma^{q-2}).
\end{equation*}
This proves~\eqref{eq-need-for-hp}, and thus completes the proof of Theorem~\ref{thm-hp-norm}.
\end{proof}

\section{Toward near-isometric duality for Hardy norms}\label{sec:duality}

Define the inner product of matrices $A, B$ by $\langle A, B\rangle = \EE_x \langle Ax, Bx\rangle$ where $x$ is uniformly distributed on $S^{n-1}$. Another way to express it is $\langle A, B\rangle = \frac{1}{n} \tr\left(A^{\mathsf T}B\right)$. This inner product gives rise to dual Hardy norms on $\mathbb R^{n\times n}$:
\begin{equation}
    \|A\|_{H^p_*} = \sup\{\langle A, B\rangle \colon \|B\|_{H^p} \le 1\}. 
\end{equation}
Of these, the $H^1_*$ norm is of particular interest, as it appears in the Schwarz lemma for harmonic mappings~\cite[Theorem 5.1]{KovalevYang}. Improving on the results of~\cite{KovalevYang}, Brevig, Ortega-Cerd\`{a} and Seip~\cite{Brevig-etal} showed that for $2\times 2$ matrices,
\begin{equation}\label{eq-2d-dual}
\|A\|_{H^4} \le \|A\|_{H^1_*} \le C_2 \|A\|_{H^4} \quad 
\text{with } C_2 = \frac{\pi}{2\sqrt[4]{6}}
\end{equation} 
where the constant $C_2 = 1.0036\ldots$ is the best possible one. 

Numerical experiments indicate that such near-isometric duality of $H^1$ and $H^4$ persists for $n>2$, except that the multiplicative constant should be placed in the lower bound. The value of this constant appears to be determined by the case $A=P_2$, a rank-two projection. 

\begin{conjecture}\label{conj-1-4} For any $n\times n$ matrix $A$ with $n>2$ we have 
\begin{equation}\label{eq-conjectural-dual}
c_n \|A\|_{H^4} \le \|A\|_{H^1_*} \le \|A\|_{H^4} \quad 
\text{where } c_n = \frac{\langle P_2, P_2\rangle }{\|P_2\|_{H^1} \|P_2\|_{H^4}}.
\end{equation} 
\end{conjecture}

We can find the explicit form of $c_n$ in~\eqref{eq-conjectural-dual}. Indeed, $\langle P_2, P_2\rangle = 2/n$ and the Hardy norms of projections can be computed as follows. 

\begin{proposition}\label{prop-projection-p-norm}
Let $1\le k\le n-1$, and let $P_k$ be an orthogonal projection of rank $k$ in $\mathbb R^n$. Then
\begin{equation}\label{eq-proj-p-norm}
\|P_k\|_{H^p}
= \left(\frac{\Gamma((k+p)/2)\, \Gamma(n/2)}{\Gamma(k/2)\,\Gamma((n+p)/2)}\right)^{1/p}.
\end{equation}
\end{proposition}

\begin{proof} Recall from the proof of Proposition~\ref{prop-projection-H0} that $|P_k x|^2$ is distributed as $\BBeta(k/2, (n-k)/2)$. The $(p/2)$-th moment of this Beta distribution is~\cite[\S25.3]{JKB2} 
\[
\frac{B((k+p)/2, (n-k)/2)}{B(k/2, (n-k)/2)} =
\frac{\Gamma((k+p)/2) \,\Gamma(n/2)}{\Gamma(k/2)\,\Gamma((n+p)/2)},
\]
which implies~\eqref{eq-proj-p-norm}.
\end{proof}

The formula~\eqref{eq-proj-p-norm} simplifies the most for projections of rank $k=n-2$, when the Gamma functional equation yields
\[
\|P_{n-2}\|_{H^p}
= \left(\frac{n-2}{n-2+p}\right)^{1/p}.
\]
For projections of rank $k=2$ the simplification is also substantial: 
\begin{equation*}\begin{split}
\|P_2\|_{H^1} & = 
\frac{\Gamma(3/2)\, \Gamma(n/2)}{\Gamma((n+1)/2)}
= \frac{\sqrt{\pi}\,\Gamma(n/2)}{2\,\Gamma((n+1)/2)}; \\
\|P_2\|_{H^4} & = \left(\frac{\Gamma(3)\, \Gamma(n/2)}{\Gamma((n+4)/2)}\right)^{1/4} = 
 \left( \frac{8}{n(n+2)} \right)^{1/4}.   
\end{split} \end{equation*}
It follows that the conjectural constant $c_n$ in~\eqref{eq-conjectural-dual} is 
\begin{equation}\label{eq-conj-cn}
c_n = \frac{2^{5/4}}{\sqrt{\pi}}\, \frac{\Gamma((n+1)/2)}{\Gamma(n/2)}\,\frac{(n+2)^{1/4}}{n^{3/4}}.
\end{equation}
For example, $c_3=0.993\ldots$, and for large $n$ we have
$c_n \to 2^{3/4}/\sqrt{\pi} = 0.949\ldots$ since the Gamma ratio in~\eqref{eq-conj-cn} is asymptotic to $\sqrt{n/2}$. Thus, the constant in Conjecture~\ref{conj-1-4} stays close to $1$ in all dimensions $n>2$. 

At the time of writing, even the  (presumably easier) upper bound in~\eqref{eq-conjectural-dual}, that is, 
\begin{equation}\label{eq-conj-upper}
\frac{1}{n}\tr\left(A^{\mathsf T}B\right) 
\le \|A\|_4 \|B\|_1 \quad A, B\in \mathbb R^{n\times n}, \ n>2,
\end{equation}
remains open. The fact that~\eqref{eq-conj-upper} fails for $n=2$ (e.g., $\|P_1\|_4 \|P_1\|_1 = (3/8)^{1/4} (2/\pi)=0.498\ldots < 1/2$) indicates that the inequality is not trivial. 

\bibliographystyle{plain}
\bibliography{references}

\end{document}